\documentclass[11pt]{article}
\usepackage[a4paper,centering,hscale=0.7,vscale=0.75]{geometry}
\usepackage{mathtools}
\usepackage{amsfonts,amsthm,amssymb}
\usepackage{amscd}

\newtheorem{prop}{Proposition}[section]
\newtheorem{thm}[prop]{Theorem}
\newtheorem{lemma}[prop]{Lemma}

\theoremstyle{remark}
\newtheorem{rmk}[prop]{Remark}
\theoremstyle{definition}

\numberwithin{equation}{section}

\renewcommand{\P}{\mathbb{P}}
\newcommand{\E}{\mathbb{E}}
\renewcommand{\L}{\mathbb{L}}
\newcommand{\erre}{\mathbb{R}}
\newcommand{\enne}{\mathbb{N}}
\newcommand{\m}{\bar{\mu}}
\newcommand{\ip}[2]{\langle#1,#2\rangle}
\DeclarePairedDelimiter\norm{\lVert}{\rVert}

\title{On maximal inequalities for purely discontinuous martingales
  in infinite dimensions}

\author{Carlo Marinelli\footnote{Department of Mathematics, University
    College London, Gower Street, London WC1E 6BT, United
    Kingdom. URL: \texttt{http://goo.gl/4GKJP}} \and Michael
  R\"ockner\footnote{Fakult\"at f\"ur Mathematik, Universit\"at
    Bielefeld, Germany.}}

\date{\normalsize 9 August 2013}

\begin{document}
\maketitle

\begin{abstract}
  The purpose of this paper is to give a survey of a class of maximal
  inequalities for purely discontinuous martingales, as well as for
  stochastic integral and convolutions with respect to Poisson
  measures, in infinite dimensional spaces. Such maximal inequalities
  are important in the study of stochastic partial differential
  equations with noise of jump type.
\medskip\par\noindent
%
%
\end{abstract}

\section{Introduction}
The purpose of this work is to collect several proofs, in part
revisited and extended, of a class of maximal inequalities for
stochastic integrals with respect to compensated random measures,
including Poissonian integrals as a special case. The precise
formulation of these inequalities can be found in Sections \ref{sec:H}
to \ref{sec:conv} below. Their main advantage over the maximal
inequalities of Burkholder, Davis and Gundy is that their right-hand
side is expressed in terms of predictable ``ingredients'', rather than
in terms of the quadratic variation.
Since our main motivation is the application to stochastic partial
differential equations (SPDE), in particular to questions of
existence, uniqueness, and regularity of solutions
(cf.~\cite{cm:MF10,cm:JFA13,cm:JFA10,cm:IDAQP10,cm:EJP10}), we focus
on processes in continuous time taking values in infinite-dimensional
spaces. Corresponding estimates for finite-dimensional processes have
been used in many areas, for instance in connection to Malliavin
calculus for processes with jumps, flow properties of solutions to
SDEs, and numerical schemes for L\'evy-driven SDEs (see
e.g. \cite{BGJ,JKMP,Kun:04}). Very recent extensions to vector-valued
settings have been used to develop the theory of stochastic
integration with jumps in (certain) Banach spaces (see \cite{Dirksen}
and references therein).

We have tried to reconstruct the historical developments around this
class of inequalities (an investigation which les us to quite a few
surprises), together with relevant references, and we hope that our
account could at least serve to correct some terminology that seems not
appropriate. In fact, while we refer to Section \ref{sec:storia} below
for details, it seems important to remark already at this stage that
the estimates which we termed ``Bichteler-Jacod's inequalities'' in
our previous article \cite{cm:JFA10} should have probably more
rightfully been baptized as ``Novikov's inequalities'', in recognition
of the contribution \cite{Nov:75}.

Let us conclude this introductory section with a brief outline of the
remaining content: after fixing some notation and collecting a few
elementary (but useful) results in Section \ref{sec:prel}, we state
and prove several upper and lower bounds for purely discontinuous
Hilbert-space-valued continuous-time martingales in Section
\ref{sec:H}. We actually present several proofs, adapting,
simplifying, and extending arguments of the existing literature. The
proofs in Subsections \ref{ssec:zappa} and \ref{ssec:khin} might be,
at least in part, new. On the issue of who proved what and when,
however, we refer to the (hopefully) comprehensive discussion in
Section \ref{sec:storia}. Section \ref{sec:lq} deals with $L_q$-valued
processes that can be written as stochastic integrals with respect to
compensated Poisson random measures. Unfortunately, to keep this
survey within a reasonable length, it has not been possible to
reproduce the proof, for which we refer to the original contribution
\cite{Dirksen}. The (partial) extension to the case of stochastic
convolutions is discussed in Section \ref{sec:conv}.


\section{Preliminaries}     \label{sec:prel}
Let $(\Omega,\mathcal{F},\mathbb{F}=(\mathcal{F}_t)_{t\geq 0},\P)$ be
a filtered probability space satisfying the ``usual'' conditions, on
which all random elements will be defined, and $H$ a real (separable)
Hilbert space with 
norm $\norm{\cdot}$.
If $\xi$ is an $E$-valued random variable, with $E$ a normed space,
and $p>0$, we shall use the notation
\[
\norm[\big]{\xi}_{\L_p(E)} := \bigl(\E\norm{\xi}_E^p\bigr)^{1/p}.
\]
Let $\mu$ be a random measure on a measurable space $(Z,\mathcal{Z})$,
with dual predictable projection (compensator) $\nu$. We shall use
throughout the paper the symbol $M$ to denote a martingale of the type
$M=g \star \m$, where $\m:=\mu-\nu$ and $g$ is a vector-valued
(predictable) integrand such that the stochastic integral
\[
(g \star \m)_t := \int_{(0,t]}\!\int_Z g(s,z)\,\m(ds,dz)
\]
is well defined. We shall deal only with the case that $g$ (hence $M$)
takes values in $H$ or in an $L_q$ space. Integrals with
respect to $\mu$, $\nu$ and $\m$ will often be written in abbreviated
form, e.g. $\int_0^t g\,d\m := (g \star \m)_t$ and $\int g\,d\m := (g
\star \m)_\infty$. If $M$ is $H$-valued, the following well-known
identities hold for the quadratic variation $[M,M]$ and the Meyer
process $\ip{M}{M}$:
\[
[M,M]_T = \sum_{s \leq T} \norm{\Delta M_s}^2 = \int_0^T \norm{g}^2\,d\mu, 
\qquad
\ip{M}{M}_T = \int_0^T \norm{g}^2\,d\nu
\]
for any stopping time $T$. Moreover, we shall need the fundamental
Burkholder-Davis-Gundy's (BDG) inequality:
\[
\norm[\big]{M^*_\infty}_{\L_p} \eqsim \norm[\big]{[M,M]^{1/2}_\infty}_{\L_p}
\qquad \forall p \in [{1,\infty}[,
\]
where $M^*_\infty:=\sup_{t \geq 0} \norm{M_t}$.
An expression of the type $a \lesssim b$ means that there exists a
(positive) constant $N$ such that $a \leq Nb$. If $N$ depends on the
parameters $p_1,\ldots,p_n$, we shall write $a
\lesssim_{p_1,\ldots,p_n} b$. Moreover, if $a \lesssim b$ and $b
\lesssim a$, we shall write $a \eqsim b$.

\medskip

The following lemma about (Fr\'echet) differentiability of powers of
the norm of a Hilbert space is elementary and its proof is omitted.
\begin{lemma}
  Let $\phi:H \to \erre$ be defined as $\phi: x \mapsto
  \norm{x}^p$, with $p>0$. Then $\phi \in C^\infty(H \setminus
  \{0\})$, with first and second Fr\'echet derivatives
  \begin{align}
  \label{eq:phi'}
  \phi'(x): \eta &\mapsto p \norm{x}^{p-2} \ip{x}{\eta},\\
  \label{eq:phi''}
  \phi''(x): (\eta,\zeta) &\mapsto
  p(p-2) \norm{x}^{p-4} \ip{x}{\eta}\ip{x}{\zeta} 
  + p \norm{x}^{p-2} \ip{\eta}{\zeta}.
  \end{align}
  In particular, $\phi \in C^1(H)$ if $p > 1$, and $\phi \in
  C^2(H)$ if $p > 2$.
\end{lemma}
It should be noted that, here and in the following, for $p\in
\left[1,2\right[$ and $p \in \left[2,4\right[$, the linear form
$\norm{x}^{p-2}\ip{x}{\cdot}$ and the bilinear form
$\norm{x}^{p-4}\ip{x}{\cdot}\ip{x}{\cdot}$, respectively,
have to be interpreted as the zero form if $x=0$.

\medskip

The estimate contained in the following lemma is simple but perhaps
not entirely trivial.
\begin{lemma}
  Let $1 \leq p \leq 2$. One has, for any $x, y \in H$,
  \begin{equation}     \label{eq:furba}
  0 \leq \norm{x+y}^p - \norm{x}^p - p\norm{x}^{p-2}\ip{x}{y} 
  \lesssim_p \norm{y}^p.
  \end{equation}
\end{lemma}
\begin{proof}
  Let $x$, $y \in H$. We can clearly assume $x$, $y \neq 0$, otherwise
  \eqref{eq:furba} trivially holds. Since the function $\phi: x
  \mapsto \norm{x}^p$ is convex and Fr\'echet differentiable on $H
  \setminus \{0\}$ for all $p \geq 1$, one has
  \[
  \phi(x+y) - \phi(x) \geq \ip{\nabla\phi(x)}{y},
  \]
  hence, by \eqref{eq:phi'},
  \[
  \norm{x+y}^p - \norm{x}^p - p\norm{x}^{p-2}\ip{x}{y} \geq 0.
  \]
  To prove the upper bound we distinguish two cases: if $\norm{x} \leq
  2\norm{y}$, it is immediately seen that \eqref{eq:furba} is true; if
  $\norm{x} > 2\norm{y}$, Taylor's formula applied to the function
  $[0,1] \ni t \mapsto \norm{x + ty}^p$ implies
  \[
  \norm{x+y}^p - \norm{x}^p - p\norm{x}^{p-2}\ip{x}{y} 
  \lesssim_p \norm{x+\theta y}^{p-2} \norm{y}^2
  \]
  for some $\theta \in ]{0,1}[$ (in particular $x+\theta y \neq
  0$). Moreover, we have
  \[
  \norm{x + \theta y} \geq \norm{x}-\norm{y}
  > 2\norm{y} - \norm{y} = \norm{y},
  \]
  hence, since $p-2 \leq 0$, $\norm{x + \theta y}^{p-2} \leq
  \norm{y}^{p-2}$.
\end{proof}

For the purposes of the following lemma only, let $(X,\mathcal{A},m)$
be a measure space, and denote $L_p(X,\mathcal{A},m)$ simply by $L_p$.
\begin{lemma}     \label{lm:interp}
  Let $1 < q < p$. For any $\alpha \geq 0$, one has
  \[
  \norm{f}^\alpha_{L_q} \leq \norm{f}^\alpha_{L_2} + \norm{f}^\alpha_{L_p}
  \]
\end{lemma}
\begin{proof}
  By a well-known consequence of H\"older's inequality one has
  \[
  \norm{f}_{L_q} \leq \norm{f}^r_{L_2} \, \norm{f}^{1-r}_{L_p},
  \]
  for some $0<r<1$. Raising this to the power $\alpha$ and applying
  Young's inequality with conjugate exponents $s:=1/r$ and $s':=
  1/(1-r)$ yields
  \[
  \norm{f}^\alpha_{L_q} \leq \norm{f}^{r\alpha}_{L_2} \, 
                           \norm{f}^{(1-r)\alpha}_{L_p}
  \leq r\norm{f}^\alpha_{L_2} + (1-r)\norm{f}^\alpha_{L_p}
  \leq \norm{f}^\alpha_{L_2} + \norm{f}^\alpha_{L_p}.
  \qedhere
  \]
\end{proof}


\section{Inequalities for martingales with values in Hilbert spaces}
\label{sec:H}
The following domination inequality, due to Lenglart \cite{Lenglart}, will be used several times.
\begin{lemma}
Let $X$ and $A$ be a positive adapted right-continuous process and an increasing predictable process, respectively, such that $\E[X_T|\mathcal{F}_0] \leq \E[A_T|\mathcal{F}_0]$ for any bounded stopping time. Then one has
\[
\E (X_\infty^*)^p \lesssim_p \E A_\infty^p \qquad \forall p \in ]{0,1}[.
\]
\end{lemma}

\begin{thm}     \label{thm:bj}
Let $\alpha \in [1,2]$. One has
\begin{equation}
\E(M^*_\infty)^p \lesssim_{\alpha,p}
\begin{cases}
  \displaystyle
  \E\biggl( \int \norm{g}^\alpha\,d\nu \biggr)^{p/\alpha}
  &\quad \forall p \in \left]0,\alpha\right],\\
  \displaystyle
  \E\biggl( \int \norm{g}^\alpha\,d\nu \biggr)^{p/\alpha} +
  \E \int \norm{g}^p\,d\nu
  &\quad \forall p \in \left[\alpha,\infty\right[,
\end{cases}
\tag{\textsf{BJ}}
\end{equation}
and
\begin{equation}     \label{eq:rBJ}
\E(M^*_\infty)^p \gtrsim_{\alpha,p}
\E\biggl( \int \norm{g}^2\,d\nu \biggr)^{p/2} +
  \E \int \norm{g}^p\,d\nu
  \qquad \forall p \in \left[2,\infty\right[.
\end{equation}
\end{thm}
\noindent Sometimes we shall use the notation $\mathsf{BJ}_{\alpha,p}$
to denote the inequality \textsf{BJ} with parameters $\alpha$ and $p$.

\medskip

Several proofs of \textsf{BJ} will be given below. Before doing that,
a few remarks are in order.
Choosing $\alpha=2$ and $\alpha=p$, respectively, one obtains the
probably more familiar expressions
\[
\E(M^*_\infty)^p \lesssim_p
\begin{cases}
  \displaystyle
  \E\biggl( \int \norm{g}^2\,d\nu \biggr)^{p/2}
  &\quad \forall p \in \left]0,2\right],\\
  \displaystyle
  \E \int \norm{g}^p\,d\nu
  &\quad \forall p \in [1,2],\\
  \displaystyle
  \E\biggl( \int \norm{g}^2\,d\nu \biggr)^{p/2} +
  \E \int \norm{g}^p\,d\nu
  &\quad \forall p \in \left[2,\infty\right[.
\end{cases}
\]
In more compact notation, \textsf{BJ} may equivalent be written as
\[
\norm{M^*_\infty}_{\L_p} \lesssim_{\alpha,p}
\begin{cases}
  \norm[\big]{g}_{\L_p(L_\alpha(\nu))}
       &\quad \forall p \in \left]0,\alpha\right],\\
  \norm[\big]{g}_{\L_p(L_\alpha(\nu))} + \norm[\big]{g}_{\L_p(L_p(\nu))} 
  &\quad \forall p \in \left[\alpha,\infty\right[,
\end{cases}
\]
where
\[
\norm[\big]{g}_{\L_p(L_\alpha(\nu))} :=
\norm[\big]{ \norm{g}_{L_\alpha(\nu)} }_{\L_p},
\qquad
\norm{g}_{L_\alpha(\nu)} := 
\biggl( \int \norm{g}^\alpha \,d\nu \biggr)^{1/\alpha}.
\]
This notation is convenient but slightly abusive, as it is not
standard (nor clear how) to define $L_p$ spaces with respect to a
random measure. However, if $\mu$ is a Poisson measure, then $\nu$ is
``deterministic'' (i.e. it does not depend on $\omega \in \Omega$),
and the above notation is thus perfectly lawful. In particular, if
$\nu$ is deterministic, it is rather straightforward to see that the
above estimates imply
\[
\norm[\big]{M^*_\infty}_{\L_p} \lesssim_p
\inf_{g_1+g_2=g} \norm[\big]{g_1}_{\L_p(L_2(\nu))} 
+ \norm[\big]{g_2}_{\L_p(L_p(\nu))}
=: \norm[\big]{g}_{\L_p(L_2(\nu)) + \L_p(L_p(\nu))},
\qquad 1 \leq p \leq 2,
\]
as well as
\[
\norm[\big]{M^*_\infty}_{\L_p} \lesssim_p
\max \Bigl( \norm[\big]{g}_{\L_p(L_2(\nu))}, \norm[\big]{g}_{\L_p(L_p(\nu))}
   \Bigr) =: \norm[\big]{g}_{\L_p(L_2(\nu)) \cap \L_p(L_p(\nu))},
\qquad p \geq 2
\]
(for the notions of sum and intersection of Banach spaces see
e.g. \cite{KPS}). Moreover, since the dual space of $\L_p(L_2(\nu))
\cap \L_p(L_p(\nu))$ is $\L_{p'}(L_2(\nu)) + \L_{p'}(L_{p'}(\nu))$ for
any $p \in [{1,\infty}[$, where $1/p+1/p'=1$, by a duality argument
one can obtain the lower bound
\[
\norm[\big]{M^*_\infty}_{\L_p} \gtrsim
\norm[\big]{g}_{\L_p(L_2(\nu)) + \L_p(L_p(\nu))}
\qquad \forall p \in ]{1,2}].
\]
One thus has
\[
\norm[\big]{M^*_\infty}_{\L_p} \eqsim_p
\begin{cases}
\norm[\big]{g}_{\L_p(L_2(\nu)) + \L_p(L_p(\nu))} &\quad
\forall p \in ]{1,2}],\\
\norm[\big]{g}_{\L_p(L_2(\nu)) \cap \L_p(L_p(\nu))} &\quad
\forall p \in [{2,\infty}[.
\end{cases}
\]
By virtue of the L\'evy-It\^o decomposition and of the BDG inequality
for stochastic integrals with respect to Wiener processes, the above
maximal inequalities admit corresponding versions for stochastic
integrals with respect to L\'evy processes (cf.~\cite{HauSei2,cm:MF10}). We do
not dwell on details here.

\subsection{Proofs}
We first prove the lower bound \eqref{eq:rBJ}. The proof is taken from
\cite{cm:SEE2} (we recently learned, however, cf. Section
\ref{sec:storia} below, that the same argument already appeared in
\cite{DzhapValk}).
\begin{proof}[Proof of \eqref{eq:rBJ}]
  Since $p/2>1$, one has
  \[
  \E[M,M]_\infty^{p/2} = \E\Bigl( \sum \norm{\Delta M}^2 \Bigr)^{p/2}
  \geq \E\sum \norm{\Delta M}^p = \E\int \norm{g}^p\,d\mu
  = \E\int \norm{g}^p\,d\nu,
  \]
  as well as, since $x \mapsto x^{p/2}$ is convex,
  \[
  \E[M,M]^{p/2}_\infty \geq \E\ip{M}{M}^{p/2}_\infty
  = \E\biggl( \int \norm{g}^2\,d\nu \biggr)^{p/2},
  \]
  see e.g. \cite{LeLePr}. Therefore, recalling the BDG inequality,
  \[
  \E(M_\infty^*)^p \gtrsim \E[M,M]^{p/2}_\infty \gtrsim
  \E\biggl( \int \norm{g}^2\,d\nu \biggr)^{p/2}
  + \E\int \norm{g}^p\,d\nu.
  \qedhere
  \]
\end{proof}

We now give several alternative arguments for the upper bounds.

The first proof we present is based on It\^o's formula and Lenglart's
domination inequality. It does not rely, in particular, on the BDG
inequality, and it is probably, in this sense, the most elementary.
\begin{proof}[First proof of $\mathsf{BJ}$]
  Let $\alpha \in ]{1,2}]$, and $\phi: H \ni x \mapsto \norm{x}^\alpha
  = h(\norm{x}^2)$, with $h: y \mapsto y^{\alpha/2}$. Furthermore, let
  $(h_n)_{n\in\enne}$ be a sequence of functions of class
  $C^\infty_c(\erre)$ such that $h_n \to h$ pointwise, and define
  $\phi_n: x \mapsto h_n(\norm{x}^2)$, so that $\phi_n \in
  C^2_b(H)$ \footnote{The subscript $\cdot_c$ means ``with compact
    support'', and $C^2_b(H)$ denotes the set of twice continuously
    differentiable functions $\varphi:H \to \erre$ such that
    $\varphi$, $\varphi'$ and $\varphi''$ are bounded.}.
  It\^o's formula (see e.g. \cite{Met}) then yields
  \[
  \phi_n(M_\infty) = \int_0^\infty \phi_n'(M_-)\,dM
  + \sum \bigl( \phi_n(M_- + \Delta M) - \phi_n(M_-) 
         - \phi'_n(M_-)\Delta M \bigr).
  \]
  Taking expectation and passing to the limit as $n \to \infty$, one
  has, by estimate \eqref{eq:furba} and the dominated convergence
  theorem,
  \begin{align*}
    \E\norm{M_\infty}^\alpha &\leq \E \sum \bigl( 
       \norm{M_- + \Delta M}^\alpha - \norm{M_-}^\alpha
       - \alpha\norm{M_-}^{\alpha-2}\ip{M_-}{\Delta M} \bigr)\\
    &\lesssim_\alpha \E \sum \norm{\Delta M}^\alpha
    = \E \int \norm{g}^\alpha\,d\mu = \E \int \norm{g}^\alpha\,d\nu,
  \end{align*}
  which implies, by Doob's inequality,
  \[
  \E(M^*_\infty)^\alpha \lesssim_\alpha \E \int \norm{g}^\alpha\,d\nu.
  \]
  If $\alpha=1$ we cannot use Doob's inequality, but we can
  argue by a direct calculation:
  \begin{align*}
    \E M^*_\infty = \E\sup_{t \geq 0} \norm[\bigg]{\int_0^t g\,d\m} &\leq
    \E\sup_{t \geq 0} \norm[\bigg]{\int_0^t g\,d\mu}
    + \E\sup_{t \geq 0} \norm[\bigg]{\int_0^t g\,d\nu}\\
    &\leq \E\sup_{t \geq 0} \int_0^t \norm{g}\,d\mu
    + \E\sup_{t \geq 0} \int_0^t \norm{g}\,d\nu\\
    &\leq 2\E\int \norm{g}\,d\nu.
  \end{align*}
  An application of Lenglart's domination inequality finishes the
  proof of the case $\alpha\in [1,2]$, $p \in ]{0,\alpha}]$.

  Let us now consider the case $\alpha=2$, $p>2$. We apply It\^o's
  formula to a $C^2_b$ approximation of $x \mapsto \norm{x}^p$, as in
  the first part of the proof, then take expectation and pass to the
  limit, obtaining
  \[
  \E\norm{M_\infty}^p \leq 
  \E \sum \bigl( \norm{M_-+\Delta M}^p - \norm{M_-}^p 
  - p \norm{M_{-}}^{p-2} \ip{M_{-}}{\Delta M_s} \bigr).
  \]
  Applying Taylor's formula to the function $t \mapsto \norm{x+ty}$ we
  obtain, in view of \eqref{eq:phi''},
  \begin{align*}
    &\norm{M_-+\Delta M}^p - \norm{M_-}^p 
    - p \norm{M_{-}}^{p-2} \ip{M_{-}}{\Delta M}\\
    &\qquad = \frac12 p(p-2) \norm[\big]{M_-+\theta\Delta M}^{p-4}
    \ip{M_-+\theta\Delta M}{\Delta M}^2\\
    &\qquad\quad + \frac12 p \norm[\big]{M_-+\theta\Delta M}^{p-2}
    \norm{\Delta M}^2\\
    &\qquad \leq \frac12 p(p-1) \norm[\big]{M_-+\theta\Delta
      M}^{p-2} \norm{\Delta M}^2,
  \end{align*}
  where $\theta \equiv \theta_s \in \left]0,1\right[$.  Since
  $\norm{M_-+\theta\Delta M} \leq \norm{M_-} + \norm{\Delta M}$, we
  also have
  \[
  \norm[\big]{M_- + \theta\Delta M}^{p-2} \lesssim_p
  \norm{M_-}^{p-2} + \norm{\Delta M}^{p-2} \leq (M^*_-)^{p-2}
  + \norm{\Delta M}^{p-2}.
  \]
  Appealing to Doob's inequality, one thus obtains
  \begin{align*}
    \E\bigl( M^*_\infty \bigr)^p \lesssim_p \E \norm{M_\infty}^p
    &\lesssim_p \E \sum \bigl( (M^*_-)^{p-2} \norm{\Delta M}^2 +
    \norm{\Delta M}^p
    \bigr)\\
    &= \E \int \bigl( (M^*_-)^{p-2} \norm{g}^2 + \norm{g}^p \bigr)\,d\mu\\
    &= \E \int \bigl( (M^*_-)^{p-2} \norm{g}^2 + \norm{g}^p \bigr)\,d\nu\\
    &\leq \E (M^*_\infty)^{p-2} \int \norm{g}^2 \,d\nu + \E\int
    \norm{g}^p\,d\nu.
  \end{align*}
  By Young's inequality in the form
  \[
  ab \leq \varepsilon a^{\frac{p}{p-2}} + N(\varepsilon) b^{p/2},
  \]
  we are left with
  \[
  \E (M^*_\infty)^p \leq \varepsilon N(p) \E (M^*_\infty)^p +
  N(\varepsilon,p) \E \biggl( \int \norm{g}^2 \,d\nu \biggr)^{p/2} +
  \E\int \norm{g}^p\,d\nu.
  \]
  The proof of the case $p>\alpha=2$ is completed choosing
  $\varepsilon$ small enough.

  We are thus left with the case $\alpha \in [{1,2}[$, $p>\alpha$.
  Note that, by Lemma \ref{lm:interp},
  \[
  \norm{\cdot}_{L_2(\nu)} \leq
  \norm{\cdot}_{L_2(\nu)} + \norm{\cdot}_{L_p(\nu)} \lesssim 
  \norm{\cdot}_{L_\alpha(\nu)} + \norm{\cdot}_{L_p(\nu)},
  \]
  hence the desired result follows immediately by the cases with $\alpha=2$
  proved above.
\end{proof}
\begin{rmk}
  The proof of $\mathsf{BJ}_{2,p}$, $p \geq 2$, just given is a
  (minor) adaptation of the proof in \cite{cm:JFA10}, while the other
  cases are taken from \cite{cm:SEE2}. However
  (cf.~Section~\ref{sec:storia} below), essentially the same result
  with a very similar proof was already given by Novikov
  \cite{Nov:75}. In the latter paper the author treats the
  finite-dimensional case, but the constants are explicitly
  dimension-free. Moreover, he deduces the case $p < \alpha$ from the
  case $p=\alpha$ using the extrapolation principle of Burkholder and
  Gundy \cite{BurGun}, where we used instead Lenglart's domination
  inequality. However, the proof of the latter is based on the former.
\end{rmk}

\begin{proof}[Second proof of $\mathsf{BJ}_{\alpha,p}$ ($p \leq \alpha$)]
  An application of the BDG inequality to $M$, taking into account
  that $\alpha/2 \leq 1$, yields
  \[
  \E (M_T^*)^\alpha \lesssim_\alpha
  \E \Bigl( \sum\nolimits_{\leq T} \norm{\Delta M}^2 \Bigr)^{\alpha/2} \leq
  \E \sum\nolimits_{\leq T} \norm{\Delta M}^\alpha =
  \E \bigl( \norm{g}^\alpha \star \mu \bigr)_T
  = \E \bigl( \norm{g}^\alpha \star \nu \bigr)_T
  \]
  for any stopping time $T$. The result then follows by Lenglart's
  domination inequality.
\end{proof}

We are now going to present several proofs for the case $p >
\alpha$. As seen at the end of the first $\mathsf{BJ}$, it suffices to consider the case $p>\alpha=2$.

\begin{proof}[Second proof of $\mathsf{BJ}_{2,p}$ ($p>2$)]
  Let us show that $\textsf{BJ}_{2,2p}$ holds if
  $\textsf{BJ}_{2,p}$ does: the identity
  \[ 
  [M,M] = \norm{g}^2 \star \mu = \norm{g}^2 \star \m 
  + \norm{g}^2 \star \nu,
  \]
  the BDG inequality, and $\mathsf{BJ}_{2,p}$ imply
  \begin{equation}     \label{eq:p2p}
  \begin{split}
  \E (M_\infty^*)^{2p} &\lesssim_p
  \E[M,M]_\infty^p \lesssim \E \big\lvert 
  \bigl( \norm{g}^2 \star \m \bigr)_\infty \big\rvert^p
  + \E \bigl( \norm{g}^2 \star \nu \bigr)_\infty^p\\
  &\lesssim_p \E \int \norm{g}^{2p}\,d\nu
  + \E\biggl( \int \norm{g}^4\,d\nu \biggr)^{p/2}
  + \E\biggl( \int \norm{g}^2\,d\nu \biggr)^{\frac12\,2p}\\
  &= \norm[\big]{g}^{2p}_{L_{2p}(\nu)}
     + \norm[\big]{g}^{2p}_{L_4(\nu)} + \norm[\big]{g}^{2p}_{L_2(\nu)}
  \end{split}
  \end{equation}
  Since $2 < 4 < 2p$, one has, by Lemma \ref{lm:interp},
  \[
  \norm[\big]{g}^{2p}_{L_4(\nu)} \leq \norm[\big]{g}^{2p}_{L_{2p}(\nu)} +
  \norm[\big]{g}^{2p}_{L_2(\nu)},
  \]
  which immediately implies that $\mathsf{BJ}_{2,2p}$ holds true. Let
  us now show that $\mathsf{BJ}_{2,p}$ implies $\mathsf{BJ}_{2,2p}$ also
  for any $p \in [1,2]$. Recalling that $\mathsf{BJ}_{2,p}$ does
  indeed hold for $p \in [1,2]$, this proves that $\mathsf{BJ}_{2,p}$
  holds for all $p \in [2,4]$, hence for all $p \geq 2$, thus
  completing the proof. In fact, completely similarly as above,
  one has, for any $p \in [1,2]$,
  \begin{align*}
  \E (M_\infty^*)^{2p} &\lesssim_p
  \E \big\lvert \bigl(\norm{g}^2 \star \m\bigr)_\infty \big\rvert^p
  + \E \bigl( \norm{g}^2 \star \nu \bigr)_\infty^p\\
  &\lesssim_p \E \int \norm{g}^{2p}\,d\nu
  + \E\biggl( \int \norm{g}^2\,d\nu \biggr)^{\frac12\,2p}. \qedhere
  \end{align*}
\end{proof}
\begin{rmk}
  The above proof, with $p>2$, is adapted from \cite{BJ}, where the
  authors assume $H=\erre$ and $p=2^n$, $n \in \enne$, mentioning that
  the extension to any $p \geq 2$ can be obtained by an interpolation
  argument.
\end{rmk}

\begin{proof}[Third proof of $\mathsf{BJ}_{2,p}$ ($p>2$)]
  Let $k \in \enne$ be such that $2^k \leq p < 2^{k+1}$. Applying the
  BDG inequality twice, one has
  \[
  \E\norm[\big]{(g \star \m)_\infty}^p 
  \lesssim_p \E \bigl( \norm{g}^2 \star \mu \bigr)_\infty^{p/2}
  \lesssim_p \E \big\lvert \bigl(\norm{g}^2 \star \m\bigr)_\infty 
                \big\rvert^{p/2}
  + \E \bigl( \norm{g}^2 \star \nu \bigr)_\infty^{p/2},
  \]
  where
  \[
  \E \big\lvert \bigl(\norm{g}^2 \star \m\bigr)_\infty \big\rvert^{p/2}
  \lesssim_p
  \E \bigl( \norm{g}^2 \star \mu \bigr)_\infty^{p/4} \lesssim_p
  \E \big\lvert \bigl(\norm{g}^4 \star \m\bigr)_\infty \big\rvert^{p/4}
  + \E \bigl( \norm{g}^4 \star \nu \bigr)_\infty^{p/4}.
  \]
  Iterating we are left with
  \[
  \E\norm[\big]{(g \star \m)_\infty}^p \lesssim_p
  \E\bigl( \norm{g}^{2^{k+1}} \star \mu \bigr)_\infty^{p/2^{k+1}}
  + \sum_{i=1}^k \E\biggl( \int \norm{g}^{2^i}\,d\nu \biggr)^{p/2^i},
  \]
  where, recalling that $p/2^{k+1} < 1$,
  \begin{align*}
    \E\bigl( \norm{g}^{2^{k+1}} \star \mu \bigr)_\infty^{p/2^{k+1}}
    &= \E \Bigl( \sum \norm{\Delta M}^{2^{k+1}} \Bigr)^{p/2^{k+1}}\\
    &\leq \E \sum \norm{\Delta M}^p = \E\int \norm{g}^p\,d\mu
    = \E\int \norm{g}^p\,d\nu.
  \end{align*}
  The proof is completed observing that, since $2 \leq 2^i \leq p$ for
  all $1 \leq i \leq k$, one has, by Lemma \ref{lm:interp},
  \begin{align*}
    \E \biggl( \int \norm{g}^{2^i}\,d\nu \biggr)^{p/2^i} =
    \E\norm[\big]{g}^p_{L_{2^i}(\nu)}
    &\leq \E\norm[\big]{g}^p_{L_2(\nu)} + \E\norm[\big]{g}^p_{L_p(\nu)}\\
    &= \E\biggl( \int \norm{g}^2\,d\nu \biggr)^{p/2}
       + \E \int \norm{g}^p\,d\nu.
  \qedhere
  \end{align*}
\end{proof}
\begin{rmk}
  The above proof, which can be seen as a variation of the previous
  one, is adapted from \cite[Lemma~4.1]{ProTal-Euler} (which was
  translated to the $H$-valued case in \cite{cm:MF10}). In
  \cite{ProTal-Euler} the interpolation step at the end of the proof
  is obtained in a rather tortuous (but interesting way), which is
  not reproduced here.
\end{rmk}

The next proof is adapted from \cite{JKMP}.
\begin{proof}[Fourth proof of $\mathsf{BJ}_{2,p}$ ($p>2$)]
  Let us start again from the BDG inequality:
  \[
  \E (M_\infty^*)^p \lesssim_p \E[M,M]_\infty^{p/2}.
  \]
  Since $[M,M]$ is a real, positive, increasing, purely discontinuous
  process with $\Delta [M,M]=\norm{\Delta M}^2$, one has
  \begin{align*}
    [M,M]_\infty^{p/2} &= \sum \bigl( [M,M]^{p/2} - [M,M]_{-}^{p/2} \bigr)\\
    &= \sum \Bigl( \bigl( [M,M]_{-} + \norm{\Delta M}^2 \bigr)^{p/2}
       - [M,M]_{-}^{p/2} \Bigr).
  \end{align*}
  For any $a$, $b \geq 0$, the mean value theorem applied to the
  function $x \mapsto x^{p/2}$ yields the inequality
  \[
  (a+b)^{p/2} - a^{p/2} = (p/2) \xi^{p/2-1} b \leq (p/2)
  (a+b)^{p/2-1}b \leq (p/2) 2^{p/2-1}(a^{p/2-1}b+b^{p/2}),
  \]
  where $\xi \in \left]a,b\right[$, hence also
  \[
  \bigl( [M,M]_- + \norm{\Delta M}^2 \bigr)^{p/2} - [M,M]_-^{p/2}
  \lesssim_p [M,M]_-^{p/2-1} \norm{\Delta M}^2 + \norm{\Delta M}^p.
  \]
  This in turn implies
  \begin{align*}
    \E[M,M]_\infty^{p/2} &\lesssim_p
    \sum \Bigl( [M,M]_-^{p/2-1} \norm{\Delta M}^2 + \norm{\Delta M}^p \Bigr)\\
    &= \E \int \Bigl( [M,M]_-^{p/2-1} \norm{g}^2 + \norm{g}^p \Bigr) \,d\mu\\
    &= \E \int \Bigl( [M,M]_-^{p/2-1} \norm{g}^2 + \norm{g}^p \Bigr) \,d\nu\\
    &\leq \E [M,M]_\infty^{p/2-1} \int \norm{g}^2\,d\nu + \E \int
    \norm{g}^p \,d\nu.
  \end{align*}
  By Young's inequality in the form
  \[
  a^{p/2-1}b \leq \varepsilon a^{p/2} + N(\varepsilon) b^{p/2}, \qquad
  a,\,b \geq 0,
  \]
  one easily infers
  \[
  \E[M,M]_\infty^{p/2} \lesssim_p \E \biggl( \int \norm{g}^2 \,d\nu
  \biggr)^{p/2} + \E \int \norm{g}^p \,d\nu,
  \]
  thus concluding the proof.
\end{proof}

\subsection{A (too?) sophisticated proof}
\label{ssec:zappa} 
In this subsection we prove a maximal inequality valid for any
$H$-valued local martingale $M$ (that is, we do not assume that $M$ is
purely discontinuous), from which $\mathsf{BJ}_{2,p}$, $p>2$, follows
immediately.
\begin{thm}     \label{thm:zappa}
  Let $M$ be \emph{any} local martingale with values in $H$. One has,
  for any $p \geq 2$,
  \[
  \E (M^*_\infty)^p \lesssim_p \E\ip{M}{M}_\infty^{p/2} + \E\bigl(
  (\Delta M)^*_\infty \bigr)^p.
  \]
\end{thm}
\begin{proof}
  We are going to use Davis' decomposition (see \cite{Mey:dual} for a
  very concise proof in the case of real martingales, a detailed
  ``transliteration'' of which to the case of Hilbert-space-valued
  martingales can be found in \cite{cm:BDG}): setting $S:=(\Delta M)^*$,
  one has $M=L+K$, where $L$ and $K$ are martingales satisfying the
  following properties:
  \begin{itemize}
    \setlength{\itemsep}{0pt}
  \item[(i)] $\norm{\Delta L} \lesssim S_-$;
  \item[(ii)] $K$ has integrable variation and
    $K=K^1+\widetilde{K^1}$, where $\widetilde{K^1}$ is the
    predictable compensator of $K^1$ and $\int |dK^1| \lesssim
    S_\infty$.
  \end{itemize}
  Since $M^* \leq L^* + K^*$, we have
  \[
  \norm{M^*_\infty}_{\L_p} \leq \norm{L^*_\infty}_{\L_p} +
  \norm{K^*_\infty}_{\L_p},
  \]
  where, by the BDG inequality, $\norm{K^*_\infty}_{\L_p} \lesssim_p
  \norm{[K,K]^{1/2}}_{\L_p}$. Moreover, by the maximal inequality for
  martingales with predictably bounded jumps in
  \cite[p.~37]{LeLePr}\footnote{One can verify that the proof in
    \cite{LeLePr} goes through without any change also for
    Hilbert-space-valued martingales.} and the elementary estimate
  $\ip{L}{L}^{1/2} \leq \ip{M}{M}^{1/2} + \ip{K}{K}^{1/2}$, one has
  \begin{align*}
    \norm{L^*_\infty}_{\L_p} &\lesssim_p
    \norm{\ip{L}{L}^{1/2}_\infty}_{\L_p}
    + \norm{S_\infty}_{\L_p}\\
    &\leq \norm{\ip{M}{M}^{1/2}_\infty}_{\L_p} +
    \norm{\ip{K}{K}^{1/2}_\infty}_{\L_p} + \norm{(\Delta
      M)^*_\infty}_{\L_p}.
  \end{align*}
  Since $p \geq 2$, the inequality between moments of a process and of
  its dual predictable projection in \cite[Theoreme 4.1]{LeLePr}
  yields $\norm{\ip{K}{K}^{1/2}}_{\L_p} \lesssim_p
  \norm{[K,K]^{1/2}}_{\L_p}$.  In particular, we are left with
  \[
  \norm{M^*_\infty}_{\L_p} \lesssim_p
  \norm{\ip{M}{M}^{1/2}_\infty}_{\L_p} 
  + \norm{(\Delta M)^*_\infty}_{\L_p} 
  + \norm{[K,K]^{1/2}_\infty}_{\L_p}.
  \]
  Furthermore, applying a version of Stein's inequality between
  moments of a process and of its predictable projection (see
  e.g. \cite{cm:BDG}, and \cite[p.~103]{Stein-LP} for the original formulation), one has, for $p \geq 2$,
  \[
  \norm{[\widetilde{K^1},\widetilde{K^1}]^{1/2}}_{\L_p} \lesssim_p
  \norm{[K^1,K^1]^{1/2}}_{\L_p},
  \]
  hence, recalling property (ii) above and that the quadratic
  variation of a process is bounded by its first variation, we are
  left with
  \begin{align*}
    \norm{[K,K]^{1/2}}_{\L_p} &\leq \norm{[K^1,K^1]^{1/2}}_{\L_p}
    + \norm{[\widetilde{K^1},\widetilde{K^1}]^{1/2}}_{\L_p}\\
    &\lesssim_p \norm{[K^1,K^1]^{1/2}}_{\L_p} \leq \norm[\bigg]{\int
      |dK^1|}_{\L_p} \lesssim \norm{(\Delta M)^*_\infty}_{\L_p}.
  \qedhere
  \end{align*}
\end{proof}

It is easily seen that Theorem \ref{thm:zappa} implies
$\mathsf{BJ}_{2,p}$ (for $p \geq 2$): in fact, one has
\[
\E\bigl( (\Delta M)^* \bigr)^p \leq 
\E \sum\norm{\Delta M}^p = \E \int \norm{g}^p \,d\nu.
\]
\begin{rmk}
  The above proof is a simplified version of an argument from
  \cite{cm:SEE2}. As we recently learned, however, a similar argument
  was given in \cite{DzhapValk}. As a matter of fact, their proofs is
  somewhat shorter than ours, as they claim that
  $[K,L]=0$. Unfortunately, we have not been able to prove this claim.
\end{rmk}

\subsection{A conditional proof}     \label{ssec:khin}
The purpose of this subsection is to show that if $\mathsf{BJ}_{2,p}$,
$p \geq 2$, holds for real (local) martingales, then it also holds for
(local) martingales with values in $H$. For this we are going to use
Khinchine's inequality: let $x \in H$, and $\{e_k\}_{k\in\enne}$ be an
orthonormal basis of $H$. Setting $x_k:=\ip{x}{e_k}$. Then one has
\[
\norm{x} = \Bigl( \sum_k x_k^2 \Bigr)^{1/2} 
= \norm[\Big]{\sum_k x_k\varepsilon_k}_{L_2(\bar{\Omega})}
\eqsim \norm[\Big]{\sum_k x_k\varepsilon_k}_{L_p(\bar{\Omega})},
\]
where $(\bar{\Omega},\bar{\mathcal{F}},\bar{\P})$ is an auxiliary
probability space, on which a sequence $(\varepsilon_k)$ of i.i.d
Rademacher random variables are defined.

Writing
\[
M_k:=\ip{M}{e_k} = g_k \star \m, \qquad g_k := \ip{g}{e_k},
\]
one has $\sum_k M_k \varepsilon_k = \Bigl(\sum_k g_k \varepsilon_k
\Bigr) \star \m$, hence Khinchine's inequality, Tonelli's theorem, and
Theorem \ref{thm:bj} for real martingales yield
\begin{align*}
  \E\norm{M}^p &\eqsim \E \norm[\Big]{\Bigl(\sum g_k \varepsilon_k
  \Bigr) \star \m}^p_{L_p(\bar{\Omega})}
  = \bar{\E} \, \E \biggl|
  \Bigl(\sum g_k \varepsilon_k \Bigr) \star \m \biggr|^p\\
  &\lesssim_p \bar{\E} \, \E \biggl( \int \Bigl\vert
  \sum g_k\varepsilon_k \Bigr\vert^2 \,d\nu \biggr)^{p/2} +
  \bar{\E} \, \E \int \Bigl\vert \sum g_k\varepsilon_k \Bigr\vert^p \,d\nu\\
  &=: I_1 + I_2.
\end{align*}
Tonelli's theorem, together with Minkowski's and Khinchine's
inequalities, yield
\begin{align*}
I_1 = \E\,\bar{\E} \biggl(\int \Bigl\vert 
\sum g_k\varepsilon_k \Bigr\vert^2 \,d\nu \biggr)^{p/2}
&= \E\norm[\bigg]{%
\int \Bigl\vert \sum g_k\varepsilon_k \Bigr\vert^2 \,d\nu%
}^{p/2}_{L_{p/2}(\bar{\Omega})}\\
&\leq \E \biggl( \int \norm[\Big]{%
\Bigl\vert \sum g_k\varepsilon_k \Bigr\vert^2}_{L_{p/2}(\bar{\Omega})}\,d\nu
\biggr)^{p/2}\\
&= \E \biggl( \int \norm[\Big]{\sum g_k\varepsilon_k}^2_{L_p(\bar{\Omega})}\,d\nu
\biggr)^{p/2}
\eqsim \biggl( \int \norm{g}^2 \,d\nu \biggr)^{p/2}.
\end{align*}
Similarly, one has
\[
I_2 = \E \int \norm[\Big]{\sum g_k\varepsilon_k}^p_{L_p(\bar{\Omega})}\,d\nu
\eqsim \E \int \norm{g}^p \,d\nu.
\]
The proof is completed appealing to Doob's inequality.\hfill$\qed$
\begin{rmk}
  This conditional proof has probably not appeared in published form,
  although the idea is contained in \cite{Knoche-diss}.
\end{rmk}


\section{Inequalities for Poisson stochastic integrals with values in
  $L_q$ spaces}     \label{sec:lq}
Even though there exist in the literature some maximal inequalities
for stochastic integrals with respect to compensated Poisson random
measures and Banach-space-valued integrands, here we limit ourselves
to reporting about (very recent) two-sided estimates in the case of
$L_q$-valued integrands.
Throughout this section we assume that $\mu$ is a Poisson random
measure, so that its compensator $\nu$ is of the form $\mathrm{Leb}
\otimes \nu_0$, where $\mathrm{Leb}$ stands for the one-dimensional
Lebesgue measure and $\nu_0$ is a (non-random) $\sigma$-finite measure on $Z$.  Let
$(X,\mathcal{A},n)$ be a measure space, and denote $L_q$ spaces on $X$
simply by $L_q$, for any $q \geq 1$. Moreover, let us introduce the
following spaces, where $p_1,p_2,p_3 \in [{1,\infty}[$:
\[
L_{p_1,p_2,p_3} := \L_{p_1}(L_{p_2}(\erre_+ \times Z \to L_{p_3}(X))),
\qquad
\tilde{L}_{p_1,p_2} := \L_{p_1}(L_{p_2}(X \to L_2(\erre_+ \times Z))).
\]
Then one has the following result, due to Dirksen \cite{Dirksen}:
\[
\norm[\big]{\sup_{t \geq 0} \norm{(g \star \m)_t}_{L_q}}_{\L_p} 
\eqsim_{p,q} \norm{g}_{\mathcal{I}_{p,q}},
\]
where
\begin{equation}     \label{eq:ipq}
\mathcal{I}_{p,q} := 
\begin{cases}
L_{p,p,q} + L_{p,q,q} + \tilde{L}_{p,q}, &\quad 1 < p \leq q \leq 2,\\
(L_{p,p,q} \cap L_{p,q,q}) + \tilde{L}_{p,q}, &\quad 1 < q \leq p \leq 2,\\
L_{p,p,q} \cap (L_{p,q,q} + \tilde{L}_{p,q}), &\quad 1 < q < 2 \leq p,\\
L_{p,p,q} + (L_{p,q,q} \cap \tilde{L}_{p,q}), &\quad 1 < p < 2 \leq q,\\
(L_{p,p,q} + L_{p,q,q}) \cap \tilde{L}_{p,q}, &\quad 2 \leq p \leq q,\\
L_{p,p,q} \cap L_{p,q,q} \cap \tilde{L}_{p,q}, &\quad 2 \leq q \leq p.
\end{cases}
\end{equation}
The proof of this result is too long to be included here. We limit
instead ourselves to briefly recalling what the main ``ingredients''
are: one first establishes extensions of the classical Rosenthal
inequality
\[
\E\Big\lvert \sum \xi_k \Big\rvert^p \lesssim_p \max \biggl(
\E \sum \lvert\xi_k\rvert^p,
\Bigl( \E \sum \lvert\xi_k\rvert^2 \Bigr)^{p/2} 
\biggr),
\]
where $p \geq 2$ and $\xi=(\xi_k)_k$ is any (finite) sequence of
independent real random variables. In particular, several extensions
are obtained in cases where the independent random variables $\xi_k$
take values in Banach spaces satisfying certain geometric properties;
further extensions are proved, by duality arguments, to the case where
$p \in ]{1,2}[$. Particularly ``nice'' versions are then derived
assuming that the random variables $\xi_k$ take values in $L_q$
spaces, thanks to their rich geometric structure. Finally, it is shown
that, using decoupling techniques, such inequalities can be extended
from sequences of independent random variables to stochastic integrals
of step processes with respect to compensated Poisson random measures.


\section{Inequalities for stochastic convolutions}
\label{sec:conv}
In this section we show how one can extend, under certain assumptions,
maximal inequalities from stochastic integrals to stochastic
convolutions using dilations of semigroups. As is well known,
stochastic convolutions are in general not semimartingales, hence
establishing maximal inequalities for them is, in general, not an easy
task. Usually one tries to approximate stochastic convolutions by
processes which can be written as solutions to stochastic differential
equations in either a Hilbert or a Banach space, for which one can
(try to) obtain estimates using tools of stochastic calculus. As a
final step, one tries to show that such estimates can be transfered to
stochastic convolutions as well, based on establishing suitable
convergence properties. At present it does not seem possible to claim
that any of the two methods is superior to the other (cf.,
e.g., the discussion in \cite{VerWei}). We choose to concentrate on
the dilation technique for its simplicity and elegance.

\medskip

We shall say that a linear operator $A$ on a Banach space $E$, such
that $-A$ is the infinitesimal generator of a strongly continuous
semigroup $S$, is of class $D$ if there exist a Banach space
$\bar{E}$, an isomorphic embedding $\iota:E \to \bar{E}$, a projection
$\pi:\bar{E} \to \iota(E)$, and a strongly continuous bounded group
$(U(t))_{t\in\erre}$ on $\bar{E}$ such that the following diagram
commutes for all $t>0$:
\[
\begin{CD}
  E @>S(t)>> E\\
@V{\iota}VV @AA{\iota^{-1}\circ\pi}A\\
  \bar{E} @>>U(t)> \bar{E}
\end{CD}
\]
As far as we know there is no general characterization of operators of
class $D$.\footnote{The definition of class $D$ is not standard and it
  is introduced just for the sake of concision.} Several sufficient
conditions, however, are known.

We begin with the classical dilation theorem by Sz.-Nagy (see
e.g. \cite{SzNagy-Foias}).
\begin{prop}
  Let $A$ be a linear operator $m$-accretive operator on a Hilbert
  space $H$. Then $A$ is of class $D$.
\end{prop}

The next result, due to Fendler \cite{Fendler}, is analogous to
Sz.-Nagy's dilation theorem in the context of $L_q$ spaces, although
it requires an extra positivity assumption.
\begin{prop}
  Let $E=L_q(X)$, where $X$ is any measure space and
  $q\in\left]1,\infty\right[$. Assume that $A$ is a linear
  $m$-accretive operator on $E$ such that $S(t):=e^{-tA}$ is
  positivity preserving for all $t>0$. Then $A$ is of class $D$, with
  $\bar{E}=L_q(Y)$, where $Y$ is another measure space.
\end{prop}

The following very recent result, due to Fr\"ohlich and Weis
\cite{FrWei}, allows one to consider classes of operators that are not
necessarily accretive (for many interesting examples, see
e.g.~\cite{VerWei}). For all unexplained notions of functional
calculus for operators we refer to, e.g., \cite{Weis:surv}.
\begin{prop}
  Let $E=L_q(X,m)$, with $q\in\left]1,\infty\right[$, and assume
  that $A$ is sectorial and admits a bounded $H^\infty$-calculus with
  $\omega_{H^\infty}(A)<\pi/2$. Then $A$ is of class $D$, and one can
  choose $\bar{E}=L_q([0,1]\times
  X,\mathrm{Leb}\otimes m)$.
\end{prop}

Let us now show how certain maximal estimates for stochastic integrals
yield maximal estimates for convolutions involving the semigroup
generated by an operator of class $D$.  Note that, since the operator
norms of $\pi$ and $U(t)$ are less than or equal to one, one has
\begin{align}
  &\E \sup_{t\geq 0} \norm[\bigg]{%
    \int_0^t\!\!\!\int_Z S(t-s) g(s,z)\,\m(ds,dz)}^p_E\nonumber\\
  &\qquad =
  \E \sup_{t\geq 0} \norm[\bigg]{%
  \pi U(t) \int_0^t\!\!\!\int_Z U(-s) \iota(g(s,z))\,\m(ds,dz)}^p_{\bar{E}}%
  \nonumber\\
  &\qquad \leq \norm{\pi}_\infty^p \; \sup_{t\geq 0} \norm{U(t)}_\infty^p \;
  \E \sup_{t\geq 0} \norm[\bigg]{%
  \int_0^t\!\!\!\int_Z U(-s) \iota(g(s,z))\,\m(ds,dz)}^p_{\bar{E}}%
  \nonumber\\
  &\qquad \leq \E \sup_{t\geq 0} \norm[\bigg]{%
  \int_0^t\!\!\!\int_Z U(-s) \iota(g(s,z))\,\bar\mu(ds,dz)}^p_{\bar{E}},
  \label{eq:sih}
\end{align}
where $\norm{\cdot}_\infty$ denotes the operator norm.
We have thus reduced the problem to finding a maximal estimate for a
stochastic integral, although involving a different integrand and on a
larger space.

If $E$ is a Hilbert space we can proceed rather easily.
\begin{prop}
  Let $A$ be of class $D$ on a Hilbert space $E$. Then one has, for
  any $\alpha \in [1,2]$,
  \[
  \E \sup_{t  \geq 0} \norm[\bigg]{%
    \int_0^t S(t-\cdot) g\,d\m}^p_E \lesssim_{\alpha,p}
  \begin{cases}
  \displaystyle
  \E\biggl( \int \norm{g}^\alpha\,d\nu \biggr)^{p/\alpha}
  &\quad \forall p \in \left]0,\alpha\right],\\
  \displaystyle
  \E\biggl( \int \norm{g}^\alpha\,d\nu \biggr)^{p/\alpha} +
  \E \int \norm{g}^p\,d\nu
  &\quad \forall p \in \left[\alpha,\infty\right[.
  \end{cases}
  \]
\end{prop}
\begin{proof}
  We consider only the case $p>\alpha$, as the other one is actually
  simpler. The estimate $\mathsf{BJ}_{\alpha,p}$ and \eqref{eq:sih}
  yield
  \begin{align*}
  &\E \sup_{t\geq 0} \norm[\bigg]{%
    \int_0^t S(t-\cdot) g\,d\m}_E^p\\
  &\qquad \lesssim_{\alpha,p}
  \E \int \norm[\big]{U(-\cdot) \iota \circ g}_{\bar{E}}^p \,d\nu
  + \E \biggl( \int \norm[\big]{U(-\cdot) \iota \circ g}_{\bar{E}}^\alpha \,d\nu
       \biggr)^{p/\alpha}\\
  &\qquad \leq
  \E \int \norm{g}_{E}^p \,d\nu
  + \E \biggl( \int \norm{g}_{E}^\alpha \,d\nu \biggr)^{p/\alpha},
  \end{align*}
  because $U$ is a unitary group and the embedding $\iota$ is
  isometric.
\end{proof}

If $E=L_q(X)$, the transposition of maximal inequalities from
stochastic integrals to stochastic convolution is not so
straightforward. In particular, \eqref{eq:sih} implies that the
corresponding upper bounds will be functions of the norms of
$U(-\,\cdot)\, \iota \circ g$ in three spaces of the type $L_{p,p,q}$,
$L_{p,q,q}$ and $\tilde{L}_{p,q}$ (with $X$ replaced by a different
measure space $Y$, so that $\bar{E}=L_q(Y)$). In analogy to the
previous proposition, it is not difficult to see that
\begin{equation}     \label{eq:ququ}
\norm[\big]{U(-\,\cdot)\, \iota \circ g}_{\L_{p_1}L_{p_2}(\erre_+\times Z \to L_{p_3}(Y))}
\leq \norm[\big]{g}_{\L_{p_1}L_{p_2}(\erre_+\times Z \to L_{p_3}(X))}.
\end{equation}
However, estimating the norm of $U(-\,\cdot)\, \iota \circ g$ in
$\tilde{L}_{p,q}(Y)$ in terms of the norm of $g$ in $\tilde{L}_{p,q}$
does not seem to be possible without further assumptions. Nonetheless,
the following sub-optimal estimates can be obtained.
\begin{prop}
  Let $A$ be of class $D$ on $E=L_q:=L_q(X)$ and $\mu$ a Poisson random
  measure. Then one has
  \[
  \E \sup_{t  \geq 0} \norm[\bigg]{%
    \int_0^t S(t-\cdot) g\,d\m}^p_{L_q} \lesssim_{p,q}
  \norm{g}_{\mathcal{J}_{p,q}},
  \]
where
  \[
  \mathcal{J}_{p,q} :=
  \begin{cases}
  L_{p,p,q} + L_{p,q,q}, &\quad 1 < p \leq q \leq 2,\\
  L_{p,p,q} \cap L_{p,q,q}, &\quad 1 < q \leq p \leq 2,\\
  L_{p,p,q} \cap L_{p,q,q}, &\quad 1 < q < 2 \leq p,\\
  L_{p,p,q} + (L_{p,q,q} \cap L_{p,2,q}), &\quad 1 < p < 2 \leq q,\\
  (L_{p,p,q} + L_{p,q,q}) \cap L_{p,2,q}, &\quad 2 \leq p \leq q,\\
  L_{p,p,q} \cap L_{p,q,q} \cap L_{p,2,q}, &\quad 2 \leq q \leq p.
  \end{cases}
  \]
\end{prop}
\begin{proof}
  Note that, if $q<2$, one has, by definition,
  $\norm{\cdot}_{\mathcal{I}_{p,q}} \leq
  \norm{\cdot}_{\mathcal{J}{p,q}}$ (where the spaces $\mathcal{I}_{p,q}$ have been defined in \eqref{eq:ipq}); if $q \geq 2$, by Minkowski's
  inequality,
  \[
  \norm[\bigg]{\biggl( \int |g|^2\,d\nu \biggr)^{1/2}}_{L_q} =
  \norm[\bigg]{\int |g|^2\,d\nu}^{1/2}_{L_{q/2}} \leq \biggl( \int
  \norm{g}^2_{L_q}\,d\nu \biggr)^{1/2},
  \]
  that is, $\norm{\cdot}_{\tilde{L}_{p,q}} \leq
  \norm{\cdot}_{L_{p,2,q}}$. This implies
  $\norm{\cdot}_{\mathcal{I}_{p,q}} \leq
  \norm{\cdot}_{\mathcal{J}{p,q}}$ for all $q \geq 2$, hence for all
  (admissible) values of $p$ and $q$. Therefore \eqref{eq:sih} and the
  maximal estimate \eqref{eq:ipq} yield the desired result.
\end{proof}

\begin{rmk}
  The above maximal inequalities for stochastic convolutions continue
  to hold if $A$ is only quasi-$m$-accretive and $g$ has compact
  support in time. In this case the inequality sign $\lesssim_{p,q}$
  has to be replaced by $\lesssim_{p,q,\eta,T}$, where $T$ is a finite
  time horizon. One simply has to repeat the same arguments using the
  $m$-accretive operator $A+\eta I$, for some $\eta>0$.
\end{rmk}


\section{Historical and bibliographical remarks}
\label{sec:storia}
In this section we try to reconstruct, at least in part, the
historical developments around the maximal inequalities presentend
above. Before doing that, however, let us briefly explain how we
became interested in the class of maximal inequalities: the
first-named author used in \cite{cm:MF10} a Hilbert-space version of a
maximal inequality in \cite{ProTal-Euler} to prove well-posedness for
a L\'evy-driven SPDE arising in the modeling of the term structure of
interest rates. The second-named author pointed out that such an
inequality, possibly adapted to the more general case of integrals
with respect to compensated Poisson random measures (rather than with
respect to L\'evy processes), was needed to solve a problem he was
interested in, namely to establish regularity of solutions to SPDEs
with jumps with respect to initial conditions: our joint efforts led
to the results in \cite{cm:JFA10}, where we proved a slightly less
general version of the inequality $\mathsf{BJ}_{2,p}$, $p \geq 2$,
using an argument involving only It\^o's formula. At the time of
writing \cite{cm:JFA10} we did not realize that, as demonstrated in
the present paper, it would have been possible to obtain the same
result adapting one of the two arguments (for L\'evy-driven integrals)
we were aware of, i.e. those in \cite{BJ} and \cite{ProTal-Euler}.

The version in \cite{cm:JFA10} of the inequality $\mathsf{BJ}_{2,p}$,
$p \geq 2$, was called in that paper ``Bichteler-Jacod inequality'',
as we believed it appeared (in dimension one) for the first time in
\cite{BJ}. This is actually what we believed until a few days ago
(this explains the label \textsf{BJ}), when, after this paper as well
as the first drafts of \cite{cm:BJlq} and \cite{cm:SEE2} were
completed, we found a reference to \cite{Nov:75} in
\cite{Wood:corr}. This is one of the surprises we alluded to in the
introduction. Namely, Novikov proved (in 1975, hence well before
Bichteler and Jacod, not to mention how long before ourselves) the
upper bound $\mathsf{BJ}_{\alpha,p}$ for all values of $\alpha$ and
$p$, assuming $H=\erre^n$, but with constants that are independent of
the dimension. For this reason it seems that, if one wants to give a
name (as we do) to the inequality $\mathsf{BJ}$ and its extensions,
they should be called Novikov's inequalities.\footnote{It should be
  mentioned that there are discrete-time real-valued analogs of
  $\mathsf{BJ}_{2,p}$, $p \geq 2$, that go under the name of
  Burkholder-Rosenthal (in alphabetical but reverse chronological
  order: Rosenthal \cite{Rosenthal} proved it for sequences of
  independent random variables in 1970, then Burkholder \cite{Burk:73}
  extended it to discrete-time (real) martingales in 1973), and some
  authors speak of continuous-time Burkholder-Rosenthal
  inequalities. One may then also propose to use the expression
  Burkholder-Rosenthal-Novikov inequality, that, however, seems too
  long.}
Unfortunately Novikov's paper \cite{Nov:75} was probably not known
also to Kunita, who proved in \cite{Kun:04} (in 2004) a slightly
weaker version of $\mathsf{BJ}_{2,p}$, $p \geq 2$, in $H=\erre^n$,
also using It\^o's formula. Moreover, Applebaum \cite{App2} calls
these inequalities ``Kunita's estimates'', but, again, they are just a
version of what we called (and are going to call) Novikov's
inequality.

Even though the proofs in \cite{BGJ,BJ} are only concerned with the
real-valued case, the authors explicitly say that they knew how to get
the constant independent of the dimension (see, in particular,
\cite[Lemma 5.1 and remark 5.2]{BGJ}). The proofs in
\cite{JKMP,ProTal-Euler} are actually concerned with integrals with
respect to L\'evy processes, but the adaptation to the more general
case presented here is not difficult. Moreover, the inequalities in
\cite{BGJ,BJ,JKMP,ProTal-Euler} are of the type
\[
\E\sup_{t \leq T} \, \norm{(g \star \m)_t}^p 
\lesssim_{p,d,T} \E\int_0^T \biggl( 
\int_Z \norm{g((s,\cdot)}^2\,dm \biggr)^{p/2}ds
+ \E\int_0^T\!\!\!\int_Z \norm{g((s,\cdot)}^p\,d\nu_0\,ds,
\]
where $\mu$ is a Poisson random measure with compensator
$\nu=\mathrm{Leb} \otimes \nu_0$. Our proofs show that all their
arguments can be improved to yield a constant depending only on $p$
and that the first term on the right-hand side can be replaced by
$\E\bigl( \norm{g}^2 \star \nu \bigr)_T^{p/2}$.

Again through \cite{Wood:corr} we also became aware of the
Novikov-like inequality by Dzhaparidze and Valkeila \cite{DzhapValk},
where Theorem \ref{thm:zappa} in proved with $H=\erre$. It should be
observed that the inequality in the latter theorem is apparently more
general than, but actually equivalent to $\mathsf{BJ}$
(cf.~\cite{cm:SEE2}).

Another method to obtain Novikov-type inequalities, also in
vector-valued settings, goes through their analogs in discrete time,
i.e. the Burkholder-Rosenthal inequality. We have not touched upon
this method, as we are rather interested in ``direct'' methods in
continuous time. We refer the interested reader to the very recent
preprints \cite{Dirksen,DMvN}, as well as to \cite{Pine:94,Wood} and
references therein.

The idea of using dilation theorems to extend results from stochastic
integrals to stochastic convolutions has been introduced, to the best
of our knowledge, in \cite{HauSei}. This method has then been
generalized in various directions, see
e.g. \cite{HauSei2,cm:MF10,VerWei}. In this respect, it should be
mentioned that the ``classical'' direct approach, which goes through
approximations by regular processes and avoid dilations (here
``classical'' stands for equations on Hilbert spaces driven by Wiener
process), has been (partially) extended to Banach-space valued
stochastic convolutions with jumps in \cite{BHZ}. The former and the
latter methods are complementary, in the sense that none is more
general than the other. Furthermore, it is well known (see
e.g. \cite{PZ-libro}) that the factorization method breaks down when
applied to stochastic convolutions with respect to jump processes.



\let\oldbibliography\thebibliography
\renewcommand{\thebibliography}[1]{%
  \oldbibliography{#1}%
  \setlength{\itemsep}{-1pt}%
}

\bibliographystyle{amsplain}
\bibliography{ref}

\end{document}